\documentclass{amsart}
\usepackage{amsfonts,amssymb,amsmath,amsthm}
\usepackage{url}
\usepackage{enumerate}

\newcommand{\C}{\mathbb{C}}

\newcommand{\QQ}{\mathbb{Q}}
\newcommand{\NN}{\mathbb{N}}

\newcommand{\ima}{\hbox{Im}}
\newcommand{\id}{\hbox{id}}
\newcommand{\grif}{\hbox{Griff}}

\newcommand{\gr}{\hbox{Gr}}

\newcommand{\wt}{\widetilde}

\newcommand{\rom}{\romannumeral}

\newtheorem{theorem}{Theorem}[section]

\newtheorem{lemma}[theorem]{Lemma}
\newtheorem{corollary}[theorem]{Corollary}
\newtheorem{proposition}[theorem]{Proposition}
\newtheorem{conjecture}[theorem]{Conjecture}
\newtheorem{nonumbering}{Theorem}

\newtheorem{nonumberingc}{Corollary}

 \theoremstyle{definition}
 \newtheorem{definition}[theorem]{Definition}
 \newtheorem{convention}{Conventions}

 \theoremstyle{remark}
 \newtheorem{remark}[theorem]{Remark}

\begin{document}
\author[Robert Laterveer]
{Robert Laterveer}
\address{Institut de Recherche Math\'ematique Avanc\'ee, Universit\'e 
de Strasbourg, 7 Rue Ren\'e Des\-car\-tes, 67084 Strasbourg CEDEX, France.}
\email{robert.laterveer@math.unistra.fr}

\date{\today}



\title{A brief note concerning hard Lefschetz for Chow groups}

\begin{abstract} We formulate a conjectural hard Lefschetz property for Chow groups, and prove this in some special cases: roughly speaking, for varieties with finite--dimensional motive, and for varieties whose self--product has vanishing middle--dimensional Griffiths group. An appendix includes related statements that follow from results of Vial.
\end{abstract}
\vskip 1cm

\maketitle

\section{Introduction}

The Bloch--Beilinson conjectures can be seen as a formidable heuristic guide that predicts the structure of Chow groups of algebraic varieties, and the precise way Chow groups are influenced by singular cohomology (cf. \cite{J2}, \cite{M}, \cite{Vo}, \cite{MNP}). To get a glimpse of this heuristic, let us look at what the Bloch--Beilinson conjectures say concerning
the hard Lefschetz property on the level of Chow groups.

Let $X$ be a smooth projective variety over $\C$ of dimension $n$, equipped with an ample line bundle $L$. Let $A^jX_{\QQ}$ denote the Chow group of codimension $j$ algebraic cycles with $\QQ$ coefficients. It is expected that $A^j_{AJ}X_{\QQ}$ (the subgroup of Abel--Jacobi trivial cycles) only depends on the cohomology groups
  \[ H^{2j-2}(X,\QQ),H^{2j-3}(X,\QQ),\ldots, H^{j}(X,\QQ)\ .\]
This leads to the following expectation:

\begin{conjecture}\label{conjinj} Let $X$ be a smooth projective variety of dimension $n$, and $L$ an ample line bundle. Then intersection induces maps
  \[ \cdot L^{n-2j+2}\colon\ \ A^j_{AJ}(X)_{\QQ}\ \to\ A^{n-j+2}(X)_{\QQ}\]
  that are injective for $2j-2\le n$.
  \end{conjecture}
  
This type of conjecture is formulated and studied in \cite{Fu}. In particular, conjecture \ref{conjinj} implies a certain weak Lefschetz property for Chow groups:
if $Y\subset X$ is a smooth ample hypersurface, restriction $A^j(X)_{\QQ}\to A^j(Y)_{\QQ}$ is injective in the range $j<n/2$; such a weak Lefschetz property was conjectured in 1974 by Hartshorne \cite{Ha}. 

Unlike cohomology, Chow groups get increasingly complicated in higher codimension (as attested by the group $A^n(X)$ of $0$--cycles, which is in general ``very large'' \cite{M}, \cite{BS}; precisely: using \cite{BS}, one sees that $\cdot L\colon A^{n-1}(X)_{\QQ}\to A^n(X)_{\QQ}$ can not surject unless $p_g(X)=0$). For this reason, in general one cannot expect the surjectivity part of the hard Lefschetz theorem to carry over from cohomology to Chow groups. In the special case where $X$ has a small Hodge diamond, however, one may expect a surjectivity statement on the level of Chow groups---as we now proceed to explain.

For simplicity, let's restrict to the case of $0$--cycles. It is expected that if 
  \[H^{n}(X,\QQ), \ldots,H^{n+r-1}(X,\QQ)\] 
  are supported in codimension $1$, then $A^n(X)_{\QQ}$ is determined by
  \[ H^{2n}(X,\QQ), H^{2n-1}(X,\QQ),\ldots, H^{n+r}(X,\QQ)\ .\]
 (This expectation can be made more precise by introducing the conjectural Bloch--Beilinson filtration $F^\ast$ on $A^n$, and stipulating that the various gradeds depend on the various cohomology groups, cf. \cite{J2}.)
Thus one is led to the following expectation:

\begin{conjecture}\label{conjsurj} Let $X$ be a smooth projective variety of dimension $n$, and $L$ an ample line bundle. Suppose $H^i(X,\QQ)=N^1 H^i(X,\QQ)$ for $n\le i<n+r$. Then intersection induces surjective maps
  \[ \cdot L^r\colon\ \ A^{n-r}(X)_{\QQ}\ \to\ A^n(X)_{\QQ}\ .\]
  \end{conjecture}
  
 (Here $N^\ast$ denotes the coniveau filtration on cohomology (definition \ref{con}).) In particular, conjecture \ref{conjsurj} implies a ``weak Lefschetz--type'' property: under the hypotheses of conjecture \ref{conjsurj}, $A^n(X)_{\QQ}$ is supported on a codimension $r$ complete intersection $Y\subset X$.
 
 The main result of this note shows conjectures \ref{conjinj} and \ref{conjsurj} can be proven in some special cases:
 
 \begin{nonumbering}[=theorem \ref{main}] Suppose the Voisin standard conjecture holds. Let $X$ be a smooth projective variety of dimension $n$, and suppose

\item{(\rom1)} Either the motive of $X$ is finite--dimensional, or $\grif^n(X\times X)_{\QQ}=0$;

\item{(\rom2)} The Lefschetz standard conjecture $B(X)$ holds;

\item{(\rom3)} $H^i(X,\QQ)_{}= N^r H^i(X,\QQ)$ for all $i\in [n-r+1,n]$.

Then for any ample line bundle $L$, the map
\[\begin{split}
 \cdot L^r\colon&\ \ A^j_{AJ}(X)_{\QQ}\ \to\ A^{j+r}_{AJ}(X)_{\QQ}\ ,\\
   \end{split}   \]
is injective for $j\le r+1$, and
 \[ 
 \begin{split}
 \cdot L^r&\colon\ \ A^j(X)_{\QQ}\ \to\ A^{j+r}(X)_{\QQ}\ ,\\
      \end{split}\]
 is surjective for $j>n-2r$.
 \end{nonumbering}
  
 The Voisin standard conjecture \cite{V0} is explained in conjecture \ref{csv} below. For the notion of finite--dimensional motive, cf. \cite{Kim3} and \cite{An}; let us merely note that conjecturally all varieties have finite--dimensional motive \cite{Kim3}, and that there are quite a few varieties known to have finite--dimensional motive (cf. section 2 below).
 
  In certain cases, some of the hypotheses are automatically satisfied and the statement simplifies somewhat; for instance, there are the following corollaries:
 
 \begin{nonumberingc}[=corollary \ref{nocsv}] Let $X$ be a smooth projective 3fold, rationally dominated by a product of curves. Suppose $A^3(X)_{\QQ}$ is supported on a divisor. Then for any ample line bundle $L$, the map 
    \[\cdot L\colon\ \ A^2_{AJ}(X)_{\QQ}\ \to\ A^3_{AJ}(X)_{\QQ}\]
    is an isomorphism.
    
  (In particular, for any ample hypersurface $Y\subset X$, restriction 
   $ A^2_{AJ}(X)_{\QQ}\ \to\ A^2_{}(Y)_{\QQ}$
   is injective, and push--forward
   $A^2(Y)_{\QQ}\ \to\ A^3(X)_{\QQ}$
   is surjective.)
   \end{nonumberingc}

\begin{nonumberingc}[=corollary \ref{nocsv2}]  Let $X$ be a smooth projective variety of dimension $n$ which is a product
   \[ X= X_1\times X_2\times \cdots \times X_s\ ,\]
   where each $X_j$ is either an abelian variety, or a variety with Abel-Jacobi trivial Chow groups.
   Suppose
    $ H^i(X,\QQ)_{}= N^r H^i(X,\QQ)$ for all $i\in [n-r+1,n]$.
Then for any ample line bundle $L$ on $X$, 
\[\begin{split}
 \cdot L^r\colon&\ \ A^j_{AJ}(X)_{\QQ}\ \to\ A^{j+r}_{AJ}(X)_{\QQ}\ \\
   \end{split}   \]
is injective for $j\le r+1$, and
 \[ 
 \begin{array}[c]{cc}
 \cdot L^r\colon&\ \ A^j(X)_{\QQ}\ \to\ A^{j+r}(X)_{\QQ}\ \\
      \end{array}\]
 is surjective for $j>n-2r$.
 \end{nonumberingc}

As noted by the anonymous referee, there is some overlap with Vial's work \cite{V4}, and corollary \ref{nocsv} easily follows from results contained in \cite{V4}. Actually, using Vial's work one can prove a stronger statement; this is explained in an appendix. We are very grateful to the referee for numerous valuable suggestions, and particularly for pointing out the relevance of \cite{V4} and sketching the proof presented in the appendix.

\begin{convention} In this note, the word {\sl variety\/} will refer to a quasi--projective irreducible algebraic variety over $\C$. A {\sl subvariety\/} is a (possibly reducible) reduced subscheme which is equidimensional. The Chow group of $j$--dimensional cycles on $X$ is denoted $A_jX$; for $X$ smooth of dimension $n$ the notations $A_jX$ and $A^{n-j}X$ will be used interchangeably. The Griffiths group $\grif^j$ is the group of codimension $j$ cycles that are homologically trivial modulo algebraic equivalence. In an effort to lighten notation, we will often write $H^jX$ or $H_jX$ to designate singular cohomology $H^j(X,\QQ)$ resp. Borel--Moore homology $H_j(X,\QQ)$.
\end{convention}

\section{Preliminary}

\begin{definition}[Coniveau filtration \cite{BO}]\label{con} Let $X$ be a quasi--projective variety. The {\em coniveau filtration\/} on cohomology and on homology is defined as
  \[\begin{split}   N^c H^i(X,\QQ)&= \sum \ima\bigl( H^i_Y(X,\QQ)\to H^i(X,\QQ)\bigr)\ ;\\
                           N^c H_i(X,\QQ)&=\sum \ima \bigl( H_i(Z,\QQ)\to H_i(X,\QQ)\bigr)\ ,\\
                           \end{split}\]
   where $Y$ runs over codimension $\ge c$ subvarieties of $X$, and $Z$ over dimension $\le i-c$ subvarieties.
 \end{definition}

We recall the statement of the ``Voisin standard conjecture'' (this is \cite[Conjecture 0.6]{V0}):

\begin{conjecture}[Voisin standard conjecture]\label{csv} Let $X$ be a smooth projective variety, and $Y\subset X$ closed with complement $U$. Then the natural sequence
  \[  N^i H_{2i}(Y,\QQ)\to N^i H_{2i}(X,\QQ)\to N^i H_{2i}(U,\QQ)\to 0\]
  is exact for any $i$.
\end{conjecture}

\begin{remark} Hodge theory gives an exact sequence
  \[    \gr^W_{-2i} H_{2i}Y\cap F^{-i}\to H_{2i}X\cap F^{-i}\to \gr^W_{-2i} H_{2i}U\cap F^{-i}\to 0\ ,\]
  where $W$ denotes Deligne's weight filtration, and $F$ the Hodge filtration on $H_\ast(-,\C)$.
  Hence if the Hodge conjecture (that is, its homology version for singular varieties \cite{J}) is true, then conjecture \ref{csv} is true. 
  
What's more, this conjecture fits in very neatly with the classical standard conjectures: Voisin shows that conjecture \ref{csv} plus the algebraicity of the K\"unneth components of the diagonal is equivalent to the Lefschetz standard conjecture \cite[Proposition 1.6]{V0}.
 \end{remark}
  
\begin{remark}\label{csvtrue} Conjecture \ref{csv} is obviously true for $i\le 1$ (this follows from the truth of the Hodge conjecture for curve classes), and for $i\ge \dim Y-1$ (where it follows from the Hodge conjecture for divisors).
\end{remark} 

The main ingredient we will use in this note is Kimura's nilpotence theorem:

\begin{theorem}[Kimura \cite{Kim3}]\label{nilp} Let $X$ be a smooth projective variety of dimension $n$ with finite--dimensional motive. Let $\Gamma\in A^n(X\times X)_{\QQ}$ be a correspondence which is homologically trivial. Then there is $N\in\NN$ such that
     \[ \Gamma^{\circ N}=0\ \ \ \ \in A^n(X\times X)_{\QQ}\ .\]
\end{theorem}

 We refer to \cite{Kim3}, \cite{An}, \cite{MNP} for the definition of finite--dimensional motive. Conjecturally, any variety has finite--dimensional motive \cite{Kim3}. What mainly concerns us in the scope of this note, is that there are quite a few examples which are known to have finite--dimensional motive:
varieties dominated by products of curves \cite{Kim3}, $K3$ surfaces with Picard number $19$ or $20$ \cite{P}, surfaces not of general type with vanishing geometric genus \cite[Theorem 2.11]{GP}, Godeaux surfaces \cite{GP}, 3folds with nef tangent bundle \cite{I}, certain 3folds of general type \cite[Section 8]{Vial}, varieties of dimension $\le 3$ rationally dominated by products of curves \cite[Example 3.15]{V3}, varieties $X$ with Abel--Jacobi trivial Chow groups (i.e. $A^i_{AJ}X_{\QQ}=0$ for all $i$) \cite[Theorem 4]{V2}, products of varieties with finite--dimensional motive \cite{Kim3}.

So far, all examples of finite-dimensional motives are in the tensor subcategory generated by Chow motives of curves.

There exists another nilpotence result, which predates and prefigures Kimura's theorem:

\begin{theorem}[Voisin \cite{V9}, Voevodsky \cite{Voe}]\label{VV} Let $X$ be a smooth projective algebraic variety of dimension $n$, and $\Gamma\in A^n(X\times X)_{\QQ}$ a correspondence which is algebraically trivial. Then there is $N\in\NN$ such that
     \[ \Gamma^{\circ N}=0\ \ \ \ \in A^n(X\times X)_{\QQ}\ .\]
    \end{theorem}

\section{Main}

We proceed to prove the main result of this note. Note that we prove slightly more than the statement given in the introduction; we also consider hard Lefschetz for the Griffiths groups.

\begin{theorem}\label{main} Suppose the Voisin standard conjecture holds. Let $X$ be a smooth projective variety of dimension $n$, and suppose

\item{(\rom1)} Either the motive of $X$ is finite--dimensional, or $\grif^n(X\times X)_{\QQ}=0$;

\item{(\rom2)} The Lefschetz standard conjecture $B(X)$ holds;

\item{(\rom3)} $H^i(X,\QQ)_{}= N^r H^i(X,\QQ)$ for all $i\in [n-r+1,n]$.

Then for any ample line bundle $L$, the maps
\[\begin{split}
 \cdot L^r\colon&\ \ A^j_{AJ}(X)_{\QQ}\ \to\ A^{j+r}_{AJ}(X)_{\QQ}\ ,\\
   \cdot L^r\colon&\ \ \grif^j(X)_{\QQ}\ \to\ \grif^{j+r}(X)_{\QQ} \\
   \end{split}   \]
are injective for $j\le r+1$, and

\[ 
 \begin{array}[c]{cc}
 \cdot L^r\colon&\ \ A^j(X)_{\QQ}\ \to\ A^{j+r}(X)_{\QQ}\ ,\\
 \cdot L^r\colon&\ \ A^j_{AJ}(X)_{\QQ}\ \to\ A^{j+r}_{AJ}(X)_{\QQ}\ ,\\
 \cdot L^r\colon&\ \ \grif^{j-1}(X)_{\QQ}\ \to\ \grif^{j+r-1}(X)_{\QQ}\\
      \end{array}\]
 are surjective for $j>n-2r$.
 
 (Moreover, $L^r$ is injective on $A^j_{hom}(X)_{\QQ}$ resp. on $A^j(X)_{\QQ}$ provided $j\le\min(r+1,{n-r+1\over 2})$ resp. $j\le\min(r+1,{n-r\over 2})$.)
 
 \end{theorem}

\begin{proof}
We first consider Chow groups, and prove the injectivity statement. Since by hypothesis $B(X)$ holds, the K\"unneth components 
  \[ \pi_i\in \ima\Bigl( H^{2n-i}X\otimes H^iX\ \to\ H^{2n}(X\times X)\Bigr)\]
 are algebraic \cite{K}. Given an ample line bundle $L$, and an integer $\ell\ge 0$, we have a correspondence $L^\ell\in A^{n+\ell}(X\times X)_{\QQ}$ which acts as ``cupping with $L^\ell$''. There is the relation
   \[  L^\ell= \frac{1}{d}\ \Gamma_{\tau}\circ{}^t \Gamma_{\tau}\ \ \in A^{n+\ell}(X\times X)_{\QQ}\ ,\]
   where $\Gamma_\tau$ is the graph of the inclusion $\tau\colon Y\to X$ and $Y$ is a complete intersection of class $[Y]=dL^\ell\in A^\ell X_{\QQ}$.
   Moreover, since we suppose $B(X)$ holds, for any $i\le n$ there exist correspondences $C_i\in A^i(X\times X)_{\QQ}$ such that
   \[ C_i\circ L^{n-i}=\id\colon\ \ H^iX\ \to\ H^iX\ .\]
   
   Now we are going to use hypothesis (\rom3) of the theorem. Applying hard Lefschetz, it follows from hypothesis (\rom3) that there exists some closed codimension $r$ subvariety $Z\subset X$ supporting the cohomology groups 
   \[H^{n-r+1}X,\ldots,H^{n+r-1}X\ .\] 
   That is, for $i\in[n-r+1,n+r-1]$, the K\"unneth component $\pi_i$ is in the kernel of the restriction homomorphism
   \[  H^{2n}(X\times X)\ \to\ H^{2n}\bigl((X\times X)\setminus(Z\times Z)\bigr)\ .\]
   Using the Voisin standard conjecture (conjecture \ref{csv}), we find there exists a cycle $P_i^\prime\in A_n(Z\times Z)_{\QQ}$ such that the push-forward $P_i\in A^n(X\times X)$ (of $P_i^\prime$ to $X\times X$) equals the K\"unneth component $\pi_i$:
   \[  P_i=\pi_i\ \ \in H^{2n}(X\times X)\ \ \forall i\in[n-r+1,n+r-1]\ .\]

    \begin{lemma}\label{noaction} Let $i\in [n-r+1,n+r-1]$. Then for $j>n-r$, we have
       \[ (P_i)_\ast A^j(X)_{\QQ}=0\ .\]
     For $j\le r+1$, we have
       \[ (P_i)_\ast A^j_{AJ}(X)_{\QQ}=0\ .\]
     Moreover, 
       \[  (P_i)_\ast \grif^j(X)_{\QQ}=0\ \ \forall j\in [0,r+1]\cup [n-r,n]\ .\]
       
     \end{lemma}
    
  \begin{proof} 
  Let $\psi\colon Z\to X$ denote the inclusion, so $P_i=(\psi\times\psi)_\ast (P_i^\prime)$. There is a factorization
  \[\begin{array}[c]{ccc}
         A^j(X)_{\QQ}&\stackrel{(P_i)_\ast}{\to}& A^{j}(X)_{\QQ}\\
         \downarrow&&\uparrow\\
         A^{j}(Z)_{\QQ}&\stackrel{(P_i^\prime)_\ast}{\to}&A^{j-r}(Z)_{\QQ}\ .\\
         \end{array}\]
      This implies the lemma for reasons of dimension: the lower left group vanishes for $j>n-r$ (since $\dim Z=n-r$); the lower right group vanishes when restricted to Abel--Jacobi trivial cycles for $j\le r+1$. 
         \end{proof}   
   
   For $i\le n-r$, we choose a rational equivalence class to represent the K\"unneth component $\pi_i$ in the following way: 
   We take arbitrary lifts of $\pi_i$ and $C_i$ in $A^n(X\times X)_{\QQ}$ resp. in $A^i(X\times X)_{\QQ}$, and we define
   \[ \Pi_i:=\pi_i\circ {}^t C_i\circ L^{n-i}\ \ \in A^n(X\times X)_{\QQ}\ ,\ \  i\le n-r\ .\]
For $i\ge n+r$, we make the following choice to represent the K\"unneth component: We define
   \[ \Pi_i:=\pi_{i}\circ L^{i-n}\circ {}^t C_{2n-i}\ \ \in A^n(X\times X)_{\QQ}\ ,   \ \ i\ge n+r\ .\]

   \begin{lemma}\label{OK} We have
     \[ \Pi_i=\pi_i\ \ \in H^{2n}(X\times X)\ \ \forall i\in [0,n-r]\cup [n+r,2n]\ .\]
     \end{lemma}
     
  \begin{proof} First, consider the case $i\le n-r$. The transpose of $\Pi_i$ is
    \[ {}^t \Pi_i={}^t \bigl( \pi_i\circ {}^t C_i\circ L^{n-i} \bigr)= L^{n-i}\circ C_i\circ \pi_{2n-i}\ \ \in H^{2n}(X\times X)\]
  (as obviously ${}^t L^{n-i}={1/d}\ \ {}^t ( \Gamma_\tau\circ{}^t \Gamma_\tau)= L^{n-i}$). Hence, the action on cohomology is
    \[  ({}^t \Pi_i)_\ast  H^jX=\begin{cases}  \id&\hbox{if\ }j=2n-i\ ;\\
                                                                   0&\hbox{if\ }j\not=2n-i\ .\\
                                                                   \end{cases}\]
                               It follows that ${}^t \Pi_i=\pi_{2n-i}\in H^{2n}(X\times X)$.     
                               
   Next, suppose $i\ge n+r$. The argument is similar: 
                                                                 The transpose of $\Pi_i$ is
    \[ {}^t \Pi_i={}^t \bigl( \pi_i\circ L^{i-n}\circ {}^t C_{2n-i} \bigr)= C_{2n-i}\circ L^{i-n}\circ \pi_{2n-i}\ \ \in H^{2n}(X\times X)\ .\]
  Hence, the action on cohomology is
    \[  ({}^t \Pi_i)_\ast  H^jX=\begin{cases}  \id&\hbox{if\ }j=2n-i\ ;\\
                                                                   0&\hbox{if\ }j\not=2n-i\ .\\
                                                                   \end{cases}\]
                               It follows that ${}^t \Pi_i=\pi_{2n-i}\in H^{2n}(X\times X)$.                                     
    \end{proof}      

\begin{lemma}\label{noaction2} Let $i\ge n+r$. Then for $j\le r+1$, we have
  \[ \begin{split}   &(\Pi_i)_\ast A^j_{AJ}X_{\QQ}=0\ ;\\
                               &(\Pi_i)_\ast \grif^jX_{\QQ}=0\ .
                                \end{split}  \]
            \end{lemma}
            
         \begin{proof} Note that ${}^t C_{2n-i}\in A^{2n-i}(X\times X)_{\QQ}$
      acts 
        \[ ({}^t C_{2n-i})_\ast\colon\ A^j_{AJ}(X)_{\QQ}\ \to\ A^{j+n-i}_{AJ}(X)_{\QQ}\ .\]
      But since $j+n-i\le 1$, the group on the right vanishes.  
               \end{proof}  
               
           The above choices give us a decomposition of the diagonal
           \[       \Delta= \sum_{i=0}^{n-r} \Pi_i + \sum_{i={n-r+1}}^{n+r-1} P_i+\sum_{i=n+r}^{2n} \Pi_i\ \ \in H^{2n}(X\times X,\QQ)\ .\]
   This is an equality of cycles modulo homological equivalence.
   Now, applying one of the two nilpotence theorems (theorem \ref{nilp} if the motive is finite--dimensional, theorem \ref{VV} in case the Griffiths group vanishes), we get that there exists $N\in\NN$ such that
   \[  \Bigl( \Delta - \sum_{i=1}^{n-r} \Pi_i - \sum_{i={n-r+1}}^{n+r-1} P_i-\sum_{i=n+r}^{2n}  \Pi_i  \Bigr)^{\circ N}=0\ \ \in A^n(X\times X)_{\QQ}\ .\]

   Developing this expression (and noting that $\Delta^{\circ N}=\Delta$), we find 
    \[ \Delta=\sum_k Q_k\ \ \in A^n(X\times X)_{\QQ}\ ,\]
    where each $Q_k$ is a composition of elements $\Pi_i$ and $P_{i^\prime}$.
    For each $k$, let $Q_k^{0}$ denote the ``tail element'' of $Q_k$, i.e. we write
    \[ Q_k= Q_k^{N^\prime}\circ Q_k^{N^\prime -1}\circ\cdots\circ Q_k^{0}\ \ \in A^n(X\times X)_{\QQ}\ ,\]
    with $Q_k^{0}\not=\Delta$ (so that $N^\prime\le N$). 
          
 Now let us consider the action of $Q_k$ on $A^j_{AJ}(X)_{\QQ}$, for $j\le r+1$.
    In case $Q_k^{0}$ is a $\Pi_i$ for some $i\in[n+r,2n]$, it follows from lemma \ref{noaction2} that
      \[(Q_k)_\ast A^j_{AJ}(X)_{\QQ}=0\ .\] 
      Likewise, if $Q_k^0$ is of the form $P_i$ (for some $i\in[n-r+1,n+r-1]$), then applying lemma \ref{noaction}, we find again
      \[ (Q_k)_\ast A^j_{AJ}(X)_{\QQ}=0\ .\]   
      It follows that the only $Q_k$ acting non--trivially are those with a tail of type $\Pi_i$, $i\le n-r$. But then (looking at the definition of $\Pi_i$ for $i\le n-r$) it follows that
      \[ A^j_{AJ}(X)_{\QQ}=\Delta_\ast A^j_{AJ}(X)_{\QQ}=\bigl( (\hbox{something})\circ L^r \bigr)_\ast A^j_{AJ}(X)_{\QQ}\ \ \ \forall j\le r+1\ .\]
   The injectivity statement is now obvious.
   
 We now proceed to prove the surjectivity statement; this is done by making one small change in the above argument. We replace the correspondences $\Pi_i$ for $i\ge n+r$ by the following modification:
   \[ \Pi_i^\prime:= L^{i-n}\circ {}^t C_{2n-i}\circ \pi_i\ \ \in A^n(X\times X)_{\QQ}\ ,\ \ i\ge n+r\ .\]
 This definition implies the following (cf. lemma \ref{OK}):
 
 \begin{lemma} For $i\ge n+r$, we have
   \[\Pi_i^\prime=\pi_i\ \ \in H^{2n}(X\times X)\ .\]
  \end{lemma}
  
  We need another lemma:
  
  \begin{lemma}\label{noaction3} Let $i\le n-r$. Then for $j> n-r$, we have
  \[ \begin{split}   &(\Pi_i)_\ast A^j_{}(X)_{\QQ}=0\ ;\\
                               &(\Pi_{i})_\ast \grif^{j-1}X_{\QQ}=0\ .
                                \end{split}  \]
            \end{lemma}      

\begin{proof} This is analogous to lemma \ref{noaction2}. Let $Y\subset X$ be a dimension $i$ complete intersection, of class $[Y]=d L^{n-i}$. Then the action of
$\Pi_i$ factors
  \[   (\Pi_i)_\ast\colon\ \ A^j(X)_{\QQ}\to A^j(Y)_{\QQ}\to A^j(X)_{\QQ}\ ,\]
  from which the required vanishing follows.
  \end{proof}
  
 Now, we have a decomposition of the diagonal in a sum of cycles
   \[    \Delta= \sum_{i=0}^{n-r} \Pi_i + \sum_{i={n-r+1}}^{n+r-1} P_i+\sum_{i=n+r}^{2n} \Pi_i^\prime\ \ \in H^{2n}(X\times X)\ .\]
 Again applying one of the two nilpotence theorems, we know there exists $N\in\NN$ such that
   \[  \Bigl( \Delta - \sum_{i=1}^{n-r} \Pi_i - \sum_{i={n-r+1}}^{n+r-1} P_i-\sum_{i=n+r}^{2n}  \Pi_i^\prime  \Bigr)^{\circ N}=0\ \ \in A^n(X\times X)_{\QQ}\ .\] 
  Upon developing:
  \[ \Delta=\sum_k Q_k\ \ \in A^n(X\times X)_{\QQ}\ ,\]
    where each $Q_k$ is a composition of elements $\Pi_i$ and $P_{i^\prime}$ and $\Pi_{i^{\prime\prime}}^\prime$.  
    We now decompose each $Q_k$ as
    \[  Q_k= Q_k^0\circ Q_k^1\circ\cdots\circ Q_k^{N^\prime}\ \ \in A^n(X\times X)_{\QQ}\ ,\]
    with $Q_k^0\not=\Delta$ (and $N^\prime\le N$).    
    
    We analyze the action of $Q_k$ on $A^{j+r}(X)_{\QQ}$ for $j>n-2r$. 
    First, in case $Q_k^0=\Pi_i$ (for some $i\le n-r$) it follows from lemma \ref{noaction3} that there is no action:
    \[  (Q_k)_\ast A^{j+r}(X)_{\QQ}=0\ .\]
    Likewise, in case $Q_k^0$ is of type $P_i$ (for some $i\in[n-r+1,n+r-1]$) we find from lemma \ref{noaction} that again
    \[ (Q_k)_\ast A^{j+r}(X)_{\QQ}=0\ .\]
    It follows that the only correspondences $Q_k$ acting are those with ``head'' $Q_k^0$ of type $\Pi_i^\prime$ (for some $i\ge n+r$). Thus we can write
    \[   A^{j+r}_{}(X)_{\QQ}=\Delta_\ast A^{j+r}_{}(X)_{\QQ}=\bigl( L^r\circ(\hbox{something}) \bigr)_\ast A^{j+r}_{}(X)_{\QQ}\ \ \ \forall j> n-2r\ ,\]
    and also
    \[   A^{j+r}_{AJ}(X)_{\QQ}=\Delta_\ast A^{j+r}_{AJ}(X)_{\QQ}=\bigl( L^r\circ(\hbox{something}) \bigr)_\ast A^{j+r}_{AJ}(X)_{\QQ}\ \ \ \forall j> n-2r\ ,\]
   The surjectivity statement is now obvious.
   
   The statements for the Griffiths group are proven in the same way; details are left to the reader.
   As for the injectivity statement in parenthesis: the Abel--Jacobi maps fit into a commutative diagram
     \[  \begin{array}[c]{ccc}   A^j_{hom}(X)_{\QQ}&\stackrel{\cdot L^r}{\to}& A^{j+r}_{hom}(X)_{\QQ}\\
                                               \downarrow&&\downarrow\\
                                               J^{j}(X)_{\QQ}&\stackrel{\cdot L^r}{\to}& J^{j+r}(X)_{\QQ}\\
                                               \end{array}\]
               (where $J^\ast$ denotes the intermediate Jacobian). Under the assumption $2j-1+r\le n$, one can show (using hard Lefschetz for cohomology) that the bottom horizontal arrow is injective. The statement for $A^j$ is proven similarly, using the cycle class map.                               
   \end{proof}

\begin{remark} The assumption ``$\grif^n(X\times X)_{\QQ}=0$'' in theorem \ref{main} is mainly of theoretical interest, and not practically useful. Indeed, there are precise conjectures (based on the Bloch--Beilinson conjectures) saying how the coniveau filtration on cohomology should influence Griffiths groups \cite{J3}. Unfortunately, it seems these conjectures are not known in any non--trivial cases. For $n=2$, it is conjectured that if $H^1(X)=0$ then $\grif^2(X\times X)_{\QQ}=0$. For $n=3$, it is conjectured that if $h^{0,2}(X)=h^{0,3}(X)=0$ then $\grif^3(X\times X)_{\QQ}=0$. For $n=4$, if $h^{2,0}(X)=h^{3,0}(X)=h^{4,0}(X)=h^{2,1}(X)=0$ then $\grif^4(X\times X)_{\QQ}$ should vanish. These predictions are particular instances of \cite[Corollary 6.8]{J3}.
\end{remark}

In certain easy cases, some hypotheses can be eliminated from theorem \ref{main}:

\begin{corollary}\label{nocsv} Let $X$ be a smooth projective 3fold. Suppose
 
 \item{(\rom1)} $A_0(X)_{\QQ}$ is supported on a divisor;
 
  \item{(\rom2)} The motive of $X$ is finite--dimensional.
  
  Then for any ample line bundle $L$, the map 
    \[\cdot L\colon\ \ A^2_{AJ}(X)_{\QQ}\ \to\ A^3_{AJ}(X)_{\QQ}\]
    is an isomorphism.
  In particular, for any ample hypersurface $Y\subset X$, the restriction map
   \[ A^2_{AJ}(X)_{\QQ}\ \to\ A^2_{AJ}(Y)_{\QQ}\]
   is injective, and push--forward
   \[ A_0(Y)_{\QQ}\ \to\ A_0(X)_{\QQ}\]
   is surjective.
   \end{corollary}
   
\begin{proof} First, as is well--known \cite{BS}, hypothesis (\rom1) implies
  \[ H^3X=N^1 H^3X\ .\]
  Hypothesis (\rom1) also implies $B(X)$; this follows from \cite{A2} or \cite[Theorem 7.1]{V}. 
  Thus we are in position to apply theorem \ref{main}, once we manage to
   explain why Voisin's standard conjecture is not needed as an extra hypothesis. Looking at the proof, we see that this conjecture is only used to obtain that a certain Hodge class in $H_6(Z\times Z)$ is algebraic, where $\dim Z=2$; this is OK by the Hodge conjecture for divisors.
   \end{proof}    
   
 \begin{corollary}\label{nocsvagain} Let $X$ be a smooth projective variety of dimension $n\ge 4$, dominated by curves. Suppose
   \[  H^n(X)= N^{ \lceil \frac{n-1}{2}   \rceil} H^n(X)\ .\]
   Then for any ample $L$,
     \[  \cdot L\colon\ \ A^2_{hom}(X)_{\QQ}\ \to\ A^3(X)_{\QQ}\]
     is injective, and
     \[   \cdot L\colon\ \ A^{n-1}(X)_{\QQ}\ \to\ A^n(X)_{\QQ}\]
     is surjective.
     \end{corollary}
     
     \begin{proof} Just as in the proof of corollary \ref{nocsv}, the K\"unneth component $\pi_n$ is represented by an algebraic cycle on something of dimension $n+1$ thanks to the Hodge conjecture for divisors. This means that theorem \ref{main} applies unconditionally.   
     \end{proof}
   
 (Note that in corollary \ref{nocsvagain}, the assumption $H^n(X)= N^{ \lceil \frac{n-1}{2}   \rceil} H^n(X)$ implies (using $B(X)$) that $H^i(X)= N^{ \lceil \frac{i-1}{2}   \rceil} H^i(X)$
for all $i$ of the same parity as $n$. That is, the Hodge structures $H^n(X), H^{n-2}(X), H^{n-4}(X), \ldots$ are of level $\le 1$.)

 \begin{corollary}\label{nocsv2}  Let $X$ be a smooth projective variety of dimension $n$ which is a product
   \[ X= X_1\times X_2\times \cdots \times X_s\ ,\]
   where each $X_j$ is either an abelian variety, or a variety with Abel-Jacobi trivial Chow groups.
   Suppose
    \[ H^i(X)_{}= N^r H^i(X)\hbox{ for\ all\ } i\in [n-r+1,n]\ .\]

Then for any ample line bundle $L$ on $X$, 
\[\begin{split}
 \cdot L^r\colon&\ \ A^j_{AJ}(X)_{\QQ}\ \to\ A^{j+r}_{AJ}(X)_{\QQ}\ \\
   \end{split}   \]
is injective for $j\le r+1$, and
 \[ 
 \begin{array}[c]{cc}
 \cdot L^r\colon&\ \ A^j(X)_{\QQ}\ \to\ A^{j+r}(X)_{\QQ}\ ,\\
      \end{array}\]
 is surjective for $j>n-2r$.
 \end{corollary}

\begin{proof} The hypotheses imply that $X$ has finite--dimensional motive, and that $B(X)$ is true (\cite{K0}, \cite{K} for abelian varieties, and \cite[Theorem 7.1]{V} or \cite{A2} for varieties with AJ--trivial Chow groups). The corollary now follows from theorem \ref{main}, once we explain why Voisin's standard conjecture is not needed as extra hypothesis. Recall that in the proof of theorem \ref{main}, Voisin's standard conjecture was only used to obtain cycles $P_i^\prime\in A_n(Z\times Z)$ (for some $Z\subset X$ of codimension $r$) such that the push--forward $P_i\in A^n(X\times X)_{\QQ}$ represents the K\"unneth component $\pi_i$:
  \[   P_i=\pi_i\ \ \in H^{2n}(X\times X)\ \ \forall i \in [n-r+1,n+r-1]\ .\]
  But this is OK unconditionally, for the following reason: each $\pi_i$ can be expressed in terms of K\"unneth components of the factors $X_j$:
  \[   \pi_i= \bigoplus_{i_1+i_2+\ldots i_s=i}  \pi_{i_1}^1\times \pi_{i_2}^2\times   \cdots \times \pi_{i_s}^s\ \ \in H^{2n}(X\times X)\ ,\]
  where $\pi_{i_j}^j\in H^{2n_j-i_j}(X_j)\otimes H^{i_j}(X_j)$ and $n_j=\dim X_j$.
  
  Given a K\"unneth component $\pi_i$, for some $i\in [n-r+1,n+r-1]$, consider its summands $\pi_{i_1}^1\times\cdots\times\pi_{i_s}^s$.
  Suppose a summand satisfies
  \[  \pi_{i_1}^1\times \pi_{i_2}^2\times   \cdots \times \pi_{i_s}^s  \in  N^r H^{2n-i}X\otimes N^r H^iX\subset H^{2n}(X\times X)\ .\]
  
  Then in particular, 
  \[  \pi_{i_1}^1\times \pi_{i_2}^2\times   \cdots \times \pi_{i_s}^s  \in F^r H^{2n-i}X\otimes F^r H^iX\subset H^{2n}(X\times X)\ \]
 (where $F^\ast$ is the Hodge filtration), and hence (by multiplicativity of the Hodge filtration)
   \[    \pi_{i_j}^j\in  F^{r_j} H^{2n_j-i_j}(X_j)\otimes F^{s_j} H^{i_j}(X_j)\subset H^{2n_j}(X_j\times X_j)\ , \ \ j=1,\ldots,s\ ,\]
   with $\sum_j r_j =\sum_j s_j=r$. 
   
   We need a lemma:
   
  \begin{lemma}\label{ab} Let $X$ of dimension $n$ be either an abelian variety, or a smooth projective variety with Abel--Jacobi trivial Chow groups. Suppose a K\"unneth component $\pi_i$
  satisfies
    \[ \pi_i\in F^r H^{2n-i}X\otimes F^s H^{i}X\subset H^{2n}(X\times X)\ .\]
    Then there exist closed subvarieties $V$, $W\subset X$ of codimension $r$ resp. $s$, and a cycle $P_i^\prime\in A_n(V\times W)_{\QQ}$ such that
    \[   (\tau_V\times \tau_W)_\ast (P_i^\prime)=\pi_i\ \ \in H^{2n}(X\times X)\ \]
    (where $\tau_V, \tau_W$ denote the inclusion morphisms).
  \end{lemma} 
  
  \begin{proof} First, suppose $X$ is an abelian variety. Then $(r,s)$ must be $(n-i,0)$ (in case $i\le n$) or $(0,i-n)$ (in case $i\ge n$). In either case, one can take $V$, resp. $W$ to be a complete intersection; the existence of the cycle $P_i^\prime$ is then ensured by the validity of $B(X)$.
  
  Next, suppose $X$ has AJ--trivial Chow groups. Then we may suppose $s\ge \frac{i-1}{2}$ and $r\ge \frac{2n-i-1}{2}$, and the existence of the requisite $V$ and $W$ follows since we know the generalized Hodge conjecture holds for $X$ \cite{moi}. From Hodge theory, we find $\pi_i$ comes from a Hodge class on $V\times W$; since $\dim(V\times W)\le n+1$, this Hodge class is algebraic.
  \end{proof}

  Applying lemma \ref{ab} to the $X_j$ and taking the product, we obtain cycles $P_i^\prime$ supported in the expected codimension and representing the K\"unneth components $\pi_i$; this ends the proof.
\end{proof}

\section{Appendix: Vial's work}

As indicated by the anonymous referee, Vial's work \cite{V4} is very relevant to the hard Lefschetz conjectures stated in the introduction. Indeed, exploiting the construction of specific Chow--K\"unneth projectors in \cite{V4}, it is easy to obtain hard Lefschetz results for Chow groups.

An important difference with our theorem \ref{main} is that there is no need for the Voisin standard conjecture. The ``cost'' for this is a switch from the coniveau filtration $N^\ast$ to a variant filtration $\wt{N}^\ast$, called the niveau filtration.

\begin{definition}[Vial \cite{V4}] Let $X$ be a smooth projective variety. The {\em niveau filtration} on homology is defined as
  \[ \wt{N}^j H_i(X)=\sum_{\Gamma\in A_{i-j}(Z\times X)_{\QQ}} \ima\bigl( H_{i-2j}(Z)\to H_i(X)\bigr)\ ,\]
  where the union runs over all smooth projective varieties $Z$ of dimension $i-2j$, and all correspondences $\Gamma\in A_{i-j}(Z\times X)_{\QQ}$.
\end{definition}

The niveau filtration is included in the coniveau filtration:
  \[ \wt{N}^j H_i(X)\subset N^j H_i(X)\ .\] 
  These two filtrations are expected to coincide; indeed, Vial shows this is true if the standard conjecture $B$ is true for all varieties \cite[Proposition 1.1]{V4}.

\begin{proposition}\label{vial} Let $X$ be a smooth projective variety of dimension $n$. Suppose the following:

\item{(\rom1)} $n\le 5$;

\item{(\rom2)} $X$ has finite--dimensional motive;

\item{(\rom3)} $B(X)$ is true;

\item{(\rom4)} $H_i(X)=\wt{N}^1 H_i(X)$ for all $i\in [n-r+1,n]$.

Let $L$ be any ample line bundle. Then 

  
  \[ \begin{split}\cdot L^r\colon&\ \ A^{n-r}_{}(X)_{\QQ}\ \to\ A^n_{}(X)_{\QQ}\\
                          \cdot L^r\colon&\ \ A^{n-r}_{AJ}(X)_{\QQ}\ \to\ A^n_{AJ}(X)_{\QQ}\\    
                      \end{split}    \]
  are surjective, and
    
 \[  L^r\colon\ \ A^{2}_{AJ}(X)_{\QQ}\cap A^2_{alg}(X)_{\QQ}\ \to\ A^{2+r}_{}(X)_{\QQ} \]
  is injective.
  (Moreover,
    $ L^r\colon\ \ A^2_{hom}(X)_{\QQ}\ \to\ A^{2+r}(X)_{\QQ}$
    is injective provided $2+r<n$.)
 
   \end{proposition}

\begin{proof} (With thanks to the referee for pointing out this proof.) The point is that $X$ verifies conditions $(\ast)$ and $(\ast\ast)$ of \cite{V4}, so that \cite[Theorems 1 and 2]{V4} apply.
From \cite[Theorems 1 and 2]{V4}, we get idempotents $\Pi_{i,k}\in A^n(X\times X)_{\QQ}$ such that $\sum_{k\ge r} (\Pi_{i,k})_\ast H_i(X)=\wt{N}^r H_i(X)$.  Since the hard Lefschetz isomorphism respects the niveau filtration, we find that there are isomorphisms
  \[  L^{i-n}\colon\ \ (\Pi_{i,k})_\ast H_i(X)\ \stackrel{\cong}{\to}\ (\Pi_{2n-i,n-i+k})_\ast H_{2n-i}(X)\ \ \hbox{for\ all\ } i-n\ge 0\ .\]
  By finite--dimensionality, it follows there are isomorphisms of Chow motives
  \[ L^{i-n}\colon\ \   (X,\Pi_{i,k})\ \stackrel{\cong}{\to} (X,\Pi_{2n-i,n-i+k},i-n)\ \ \hbox{for\ all\ } i-n\ge 0\ .\]
  Taking Chow groups, this implies there are isomorphisms (for any $j$)
  \[  L^{i-n}\colon\ \ (\Pi_{i,k})_\ast A^j(X)_{\QQ}\ \stackrel{\cong}{\to}\ (\Pi_{2n-i,n-i+k})_\ast A^{j+i-n}(X)_{\QQ}\ \ \hbox{for\ all\ } i-n\ge 0\ .\]

  First, let's prove surjectivity. Using \cite[Theorem 2 point 1]{V4}, we see that
    \[ A^n_{}(X)_{\QQ}= \sum_i (\Pi_{i,0})_\ast A^n_{}(X)_{\QQ}\ .\]
   The hypothesis on $H_i(X)$ implies $\Pi_{i,0}=0\ \forall i>n-r$ by \cite[Theorem 2 point 4]{V4}, so that
     \[   A^n_{}(X)_{\QQ}= \sum_{i\le n-r} (\Pi_{i,0})_\ast A^n_{}(X)_{\QQ}\ .\]  
     But from the above remarks, we find that
     \[  \sum_{i\le n-r} L^{n-i}\colon\ \ \bigoplus_{i\le n-r} (\Pi_{2n-i,i-n})_\ast A^{i}_{}(X)_{\QQ}\ \to\ \sum_{i\le n-r} (\Pi_{i,0})_\ast A^n_{}(X)_{\QQ}= A^n_{}(X)_{\QQ}\]
     is surjective,
     hence (by mapping $A^i$ to $A^{n-r}$ via $L^{n-r-i}$ for $i<n-r$)
     \[  L^r\colon\ \ A^{n-r}(X)_{\QQ}\ \to\ A^n(X)_{\QQ}\]
     is surjective.
       
      The proof for $A^n_{AJ}$ is the same.
     
  It remains to prove injectivity. 
  We find from \cite[Theorem 2 point 1]{V4} that
     \[  A^2_{AJ,alg}(X)_{\QQ}=\sum_{k\le n-2} (\Pi_{i,k})_\ast A^2_{AJ,alg}(X)_{\QQ}\ .\]    
     Now we are repeatedly going to apply the various points of \cite[Theorem 2]{V4} to eliminate certain projectors from this sum.
 
  If $k=0$, then $i= n$ by \cite[Theorem 2 points 1, 2, 3]{V4}. But $\Pi_{n,0}=0$ by the hypothesis on $H_n(X)$ and \cite[Theorem 2 point 4]{V4}.
  
  The projectors $\Pi_{i,1}$ can likewise be eliminated: $\Pi_{n-1,1}$ and $\Pi_{n,1}$ don't act by points 2 resp. 3 from loc. cit., and $\Pi_{n+1,1}=0$ by hypothesis, provided $r\ge 2$
  (since $\gr^1_{\wt{N}} H_{n+1}(X)=\gr^0_{\wt{N}} H_{n-1}(X)=0$).
  
  Next, the projectors $\Pi_{i,2}$: for $n\le 3$ these don't act (point 1 of loc. cit), while for $n=4, 5$ we have that $\Pi_{n,2}$ doesn't act by point 5 resp. point 6 of loc. cit. The projector $\Pi_{n+1,2}$ doesn't act for $n=4$ (point 6 of loc. cit.), nor for $n=5$ (point 3 of loc. cit.). The projector $\Pi_{n+2,2}$ is $0$ by hypothesis, provided $r\ge 3$ (since $\gr^2_{\wt{N}} H_{n+2}(X)=\gr^0_{\wt{N}} H_{n-2}(X)=0$).
  
 The last case we need to check is that of $\Pi_{i,3}$. These only act when $n=5$. We have $i\ge 6$ (point 1 of loc. cit.), $i\not=6$ (point 3 of loc. cit.), $i\not=7$ (point 6 of loc. cit.). So the only projector acting is $\Pi_{8,3}$ (that is, provided $r=3$).
  
  Resuming this analysis, we find that
  \[ A^2_{AJ,alg}(X)_{\QQ}=\sum_{\substack{k\le n-2\\  i\ge n+r} } (\Pi_{i,k})_\ast A^2_{AJ,alg}(X)_{\QQ}\ .\]

    But from the above remarks, we find that
    \[ L^{r}\colon\ \ (\Pi_{i,k})_\ast A^{2}(X)_{\QQ}\ \to\ (\Pi_{i-2r,n-i+k})_\ast A^{2+r}(X)_{\QQ}\]
    is injective as soon as $i\ge n+r$.
    
 As for the injectivity statement in parentheses: looking at the above proof of injectivity, we see that the hypothesis ``Abel-Jacobi and algebraically trivial'' is only used in the extremal cases $(n,r)=(4,2)$ and $(n,r)=(5,3)$. That is, as long as $2+r<n$ we have
   \[  A^2_{hom}(X)_{\QQ}=\sum_{\substack{k\le n-2\\  i\ge n+r} } (\Pi_{i,k})_\ast A^2_{hom}(X)_{\QQ}\ ,\]   
   and injectivity follows.

    \end{proof}

\begin{corollary}\label{corvial} Let $X$ be a smooth projective variety of dimension $n\le 5$, dominated by curves. Suppose $A^n(X)_{\QQ}$ is supported on a surface.
Then for any ample line bundle $L$,
  \[  \cdot L^{n-2}\colon\ \ A^2_{AJ}(X)_{\QQ}\cap A^2_{alg}(X)_{\QQ}\ \to\ A^{n}(X)_{\QQ}\]
  is injective, and
  \[ \cdot L^{n-2}\colon\ \ A^{2}(X)_{\QQ}\ \to\ A^n(X)_{\QQ}\]
  is surjective.
  
  (In particular, $A^n(X)_{\QQ}$ is supported on a dimension $2$ complete intersection.)
\end{corollary}

\begin{proof} This follows from proposition \ref{vial}, in combination with lemma \ref{BS} below.
\end{proof}

\begin{lemma}\label{BS} Let $X$ be a smooth projective variety of dimension $n$. Suppose $A^n(X)_{\QQ}$ is supported on a surface. Then
  \[ H_i(X)=\wt{N}^1 H_i(X)\ \ \hbox{for\ all\ }i>2\ .\]
 \end{lemma}
 
 \begin{proof} (This is the same argument as \cite[Proposition 2.2]{V4}, which is the case $A^n(X)_{\QQ}$ supported on a curve.)
 Using \cite{BS}, one obtains a decomposition of the diagonal
   \[ \Delta=\Delta_1+\Delta_2\ \ \in A^n(X\times X)_{\QQ}\ ,\]
   with $\Delta_1$ supported on $D\times X$ for some divisor $D$, and $\Delta_2$ supported on $X\times S$, for $S\subset X$ a surface.
   We consider the action of the $\Delta_i$ on $\gr^0_{\wt{N}} H_i(X)$ (this is possible: \cite[Proposition 1.2]{V4}).
   The correspondence $\Delta_1$ does not act, as it factors over
   \[ H_{i-2}(\wt{D})/ \wt{N}^0 =0\ .\]
   The correspondence $\Delta_2$ does not act for $i>2$, as it factors over 
   \[  \gr^0_{\wt{N}} H_i(S)=\gr^0_{N} H_i(S)=0\ .\]
 \end{proof}

Note that corollary \ref{corvial} is considerably stronger than our corollary \ref{nocsv}, just as proposition \ref{vial} is more powerful than our theorem \ref{main}. This reflects the fact that Vial's Chow--K\"unneth projectors are far more refined than the ``K\"unneth lifts'' we use in the proof of theorem \ref{main}. 

For example, let $X$ be a variety of dimension $5$ dominated by curves. Proposition \ref{vial} gives a hard Lefschetz statement as soon as
$H_5(X)=\wt{N}^1 H_5(X)$ (i.e. the Hodge level of $H_5$ could be $3$). Theorem \ref{main} only works without assuming the Voisin standard conjecture if $H^5(X)=N^2H^5(X)$ (i.e., the Hodge level is $1$).

It seems likely corollary \ref{nocsvagain} can also be improved using the methods of \cite{V4}.


\end{document}